\documentclass[
a4paper
]{amsart}
\usepackage[utf8]{inputenc}
\usepackage{amsthm,mathtools}
\usepackage{amsmath}
\usepackage{comment}
\usepackage{amsfonts}
\usepackage{array}
\usepackage{stackrel}
\usepackage{graphicx}
\usepackage[english]{babel}

\usepackage[utf8]{inputenc}
\usepackage{subcaption}
\usepackage[export]{adjustbox}
\usepackage{mathabx}
\usepackage{wrapfig}
\usepackage[margin=1.2in]{geometry}
\usepackage{float}
\usepackage{tikz-cd,amssymb}
 \tikzcdset{every label/.append style = {font = \normalsize}}
\usepackage{adjustbox}

\newtheorem{thm}{Theorem}[section]

\newtheorem{prop}[thm]{Proposition}

\newtheorem{cor}[thm]{Corollary}

\newtheorem{question}[thm]{Question}

\newtheorem{remk}[thm]{Remark}

\title{Congruence Subgroups of the Virtual Braid Group}
\author{Wade Bloomquist, Alexa Goldberg, and Nancy Scherich}

\date{}
\begin{document}

\maketitle
\begin{abstract}
We extend the notion of congruence subgroups of the braid group to the virtual braid group using an extension of the integral Burau representation. We prove that the level 2 congruence subgroup of the virtual braid group is the pure virtual braid group, recovering a virtual analogue of a result of Arnol'd. We pose several questions which highlight the difference between the classical and virtual braid groups.

\end{abstract}

\newcommand{\Z}{\mathbb{Z}}
\newcommand{\hp}{\hat{\rho}}

\section{Introduction}

The study of braid groups has a rich history in mathematics as braids are interesting topologically, geometrically, and algebraically. From these different perspectives, braids have been generalized in many ways.  A common theme of study is to choose an attribute about the braid group and see if a generalization of the braid group also has such an attribute.
In this paper, we consider properties of congruence subgroups in a generalization of the braid group called the virtual braid group. 

Let $B_n$ denote the braid group on $n$ strands.
The integral Burau representation is a group homomorphism from $B_n$ into $GL_n(\mathbb{Z})$, which is a specialization of the Burau representation first defined by Werner Burau in 1935 \cite{Burau}.
The \emph{level $m$ congruence subgroup of $B_n$}, denoted $B_n[m]$, is the kernel of the integral Burau representation when the coefficients are reduced mod $m$.
\[B_n[m]:= \ker\bigg(B_n \xrightarrow{ Burau}{} GL_{n}(\Z[t^{\pm 1}])\xrightarrow{ t=-1}{} GL_{n}(\Z)\xrightarrow{\!\!\!\!\!\!\mod m}{} GL_{n}(\Z/m\Z) \bigg) \]
These congruence subgroups  are finite index subgroups of $B_n$ that have been well studied by many authors, e.g. A'Campo \cite{ACampo}, Arnol'd \cite{A68}, Assion \cite{Assion},  Brendle \cite{brendle}, Brendle--Margalit \cite{brendleM18}, B.-Patzt-S. \cite{BPS}, Kordek--Margalit \cite{KM19}, McReynolds \cite{McReynolds}, Nakamura \cite{Nakamura}, and Stylianakis \cite{S18}. 
 In particular, Arnol'd proved that the congruence subgroup mod $2$ is the pure braid group as a subgroup of $B_n$ \cite{A68}. Arnol'd's methods are quite technical and rely on topological methods.
 The first contribution we give in this paper is a simple and purely algebraic proof of Arnol'd's result, which can be found in Section \ref{sec:classical}.

The \emph{virtual braid group on $n$ strands}, denoted by $vB_n$, is an extension of the braid group by adding virtual crossings \cite{Kauff1, Kauff2}, just as virtual knot theory is an extension of classical knot theory.
 Vershinin showed that the integral Burau representation extends to the virtual braid group \cite{VER}, and we define the \emph{level $m$ congruence subgroup of $vB_n$}, denoted $vB_n[m]$ to be the kernel of this extended integral Burau representation with coefficients reduced mod $m$.
 We compare these new virtual congruence subgroups to the congruence subgroups in $B_n$ and recover some virtual analogues of properties known to be true for $B_n$. In particular, we prove the virtual analogue of Arnold's result using the new algebraic proof we give of Arnol'd's result for the braid group.\\

    \noindent \textbf{Main Theorem.}\emph{ The level 2 congruence subgroup of $vB_n$  is the pure virtual subgroup of $vB_n$.}\\

We end the paper with some interesting open questions to further the study congruence subgroups of $vB_n$, which highlight some differences compared to classical congruence subgroups, and an in-depth discussion of the congruence subgroups of $vB_2$.\\

In this paper we address several structural questions related to congruence subgroups of virtual braid groups.  We collect them here to highlight differences with the classical case and emphasize which questions remain open.

\begin{center}
\begin{tabular}{ | p{4.2cm} | p{3.9cm}| p{5.8cm} | } 
 \hline
 Posed Question & Answer for classical $B_n$ & Partial answers for virtual $vB_n$  \\ \hline

\textbf{Question \ref{ques:indexofprod}}:
What is the index of  $vB_n[m] vB_n[\ell]$ in $vB_n[gcd(m,\ell)]$?
 
 &
$B_n[m] B_n[\ell]$ is equal to $B_n[gcd(m,\ell)]$, so the index is $1$. \cite{S18,BPS}
 
 & In Proposition \ref{prop:relprimenot2}, we show that $vB_n[m] vB_n[\ell]=vB_n(=vB_n[1])$ if and only if both $m=2$ and $\ell$ is odd.\\ \hline

 \textbf{Question \ref{ques:whatstheimage}}: What is the image of the integral Burau representation?
    
    &This is known and can be found in \cite{BPS}(See Corollary C.).
    
    &The techniques of \cite{BPS} do not apply directly because they  rely, for example, on Question \ref{ques:indexofprod}.
     We provide an answer to this question when $n=2$ in Corollary \ref{cor:imagen=2}, showing the image is isomorphic to $D_\infty$.\\ \hline

\textbf{Question \ref{ques:semidirectAsubgroup}}:
For what values of $m$ is $vB_n/vB_n[m]$ a subgroup of $vB_n$? & 
 As $B_n$ has no torsion, $B_n/B_n[m]$ is never a subgroup of $B_n$. & 
 We show in Corollary \ref{cor:semi} that this is true for $m=2$. We show in Theorem \ref{thm:n=2notsemidirect} when $n=2$, that this is true if and only if $m=2$. \\ \hline

     \textbf{Question \ref{ques:normalclosure}:}
What is the normal closure of $B_n[m]$ inside $vB_n$? What is a normal generating set for $vB_n[m]$? &
  A normal generating set for $B_n[m]$ is fully known and described in \cite{BH}(See Corollary B.) together with a historical account of the problem.
 &
For $n=2$, in Theorem \ref{thm:notnormalclosure} we prove that $vB_2[m]$ is not the normal closure of $B_2[m]$, and we exhibit a normal generating set for $vB_2[m]$ in  Theorem \ref{thm:normalclosure}.
     We note that the normal closure of $B_n$ inside of $vB_n$ is known as the Kure virtual braid group. \\ 
 \hline
\end{tabular}
\end{center}

We note that all results are phrased in terms of the virtual braid group, but can be reformulated for the welded braid group.  This is elaborated in Remark \ref{weld}.

\bigskip

\noindent \textbf{Acknowledgments.} This project is part of an undergraduate research experience for the second author lead by the third author.
This project was partially funded by the AMS-Simons travel grant and the National Science Foundation Grant No. DMS-2532699 funding the third author. We would like to thank our anonymous referee for helpful comments and answers to several of our posed questions.

\section{Congruence Subgroups in the Classical Setting.}\label{sec:classical}

The \emph{braid group} on $n$-strands is denoted $B_n$ and is generated by $\sigma_1,\cdots, \sigma_{n-1}$. 
Each $\sigma_i$ can be visualized as a positive crossing between the $i$'th and $i+1$'st strand, as shown in Figure \ref{fig:braid_diagrams}, where the group operation is vertical stacking.
The inverse of a generator is a negative crossing between the $i$'th and $i+1$'st strand, and each $\sigma_i$ has infinite order.
There are two types of relations, far commutativity and the braid relation, also shown in Figure \ref{fig:braid_diagrams}.

\begin{figure}[htp]
\centering
\begin{picture}(220,120)
\put(-40,80){\includegraphics[width=.3\textwidth]{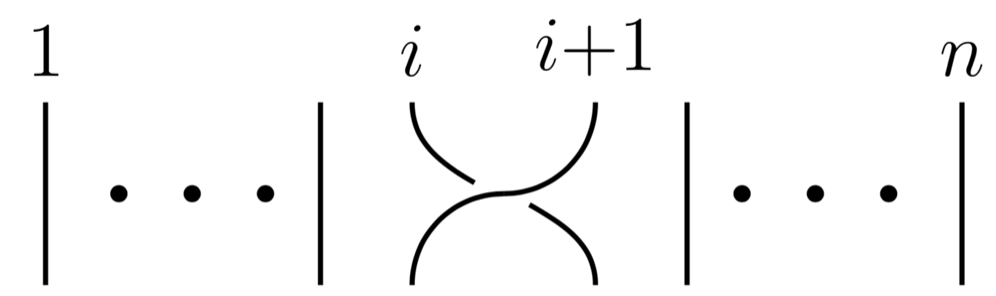}\hspace{.6cm}}
\put(-40,15){\includegraphics[width=.3\textwidth]{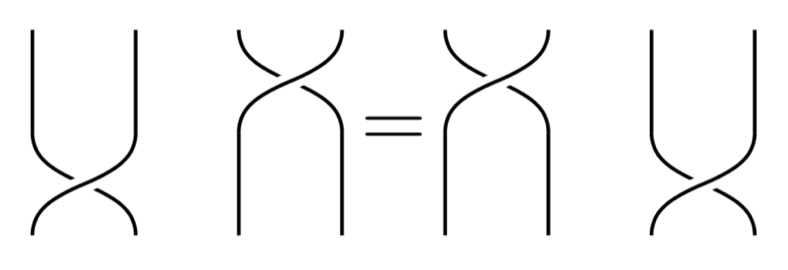}\hspace{.7cm}}
\put(130,10){\includegraphics[width=.3\textwidth]{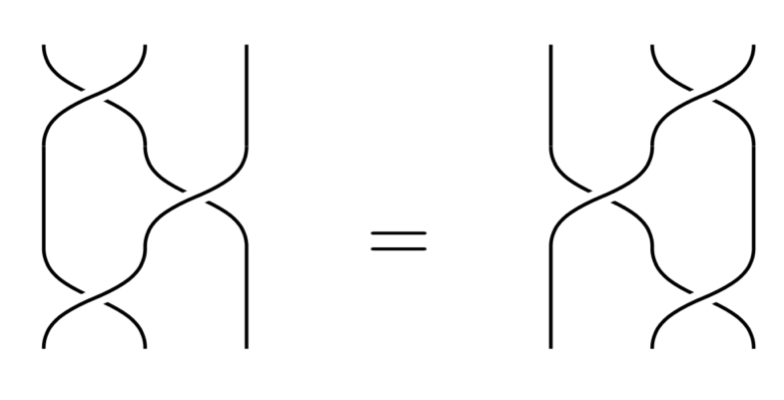}\hspace{.7cm}}
\put(-55,70){(A)}
\put(-55,5){(B)}
\put(120,5){(C)}
\put(20,75){$\sigma_i$}
\put(0,10){$\sigma_i\sigma_j=\sigma_j\sigma_i$}
\put(0,-3){for $j\neq i-1,i+1$}
\put(153,5){$\sigma_i\sigma_{i+1}\sigma_i=\sigma_{i+1}\sigma_{i}\sigma_{i+1}$}
\end{picture}
\caption{(A) The generator $\sigma_i$. (B) The far commutativity relation. (C) The braid relation.}  \label{fig:braid_diagrams}
\end{figure}
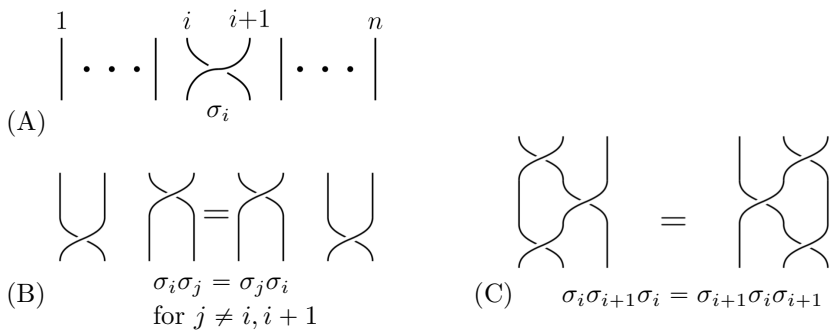

Let $S_n$ denote the symmetric group of size $n!$. The group $S_n$ can be generated by transpositions $\tau_1,\cdots, \tau_{n-1}$, where $\tau_i=(i ,i+1)$ is the order 2 element that swaps $i$ and $i+1$.
The $\tau_i$'s satisfy the far commutativity and braid relation just as the $\sigma_i$'s do.
There is a surjective map $\pi:B_n\rightarrow S_n$ that maps $\sigma_i\mapsto \tau_i$. This map can be thought of as sending a braid to the induced permutation given by following the endpoints of the braided strands.
The kernel of this map $\pi$ is denoted $P_n$ and is called the \emph{pure braid group} as a normal subgroup of $B_n$.
Visually, pure braids are braids where each strand starts and ends in the same horizontal position.

The (unreduced) \emph{Burau Representation} is a group homomorphism 
$\rho :B_n\rightarrow GL_{n}(\mathbb{Z}[t,t^{-1}])$ given by 
  \begin{equation*}
\rho(\sigma _i) =
\left (\begin{array}{c|cc|c}
I_{i-1} & 0 &0 &0 \\ \hline
0 & 1-t & t  &0\\
0&1 & 0 &0 \\ \hline
0 &0& 0& I_{n-i-1} \\
\end{array} \right )
\end{equation*} 

\noindent for each $i=1,2,\cdots, n-1$ and where $I_k$ is the $k\times k$ identity matrix.
 Specializing $t=-1$ into the Burau representation is called the integral Burau representation and has image in $GL_n(\mathbb{Z})$.
 The composition of $\rho$ with specializing $t=-1$ will be denoted by $\rho_-$, as shown in Figure \ref{fig:comdiag}.
  The integral Burau representation has found applications across mathematics in areas such as algebraic geometry, number theory, dynamics, and topology.  We refer the reader to Brendle-Margalit and the references within \cite{brendleM18} for more details about the integral Burau representation.
 Specializing $t=1$ factors through the map $\pi$ and gives the permutation representation of $S_n$ as permutation matrices in $GL_n(\mathbb{Z})$.
 This specialization of Burau is denoted $\rho_+$.  Lastly, let $r_m$ be the map that reduces the integer coefficients mod $m$, and composing $r_m\circ \rho_{\pm}$ yields representations of $B_n$ into $GL_n(\mathbb{Z}_m)$.
 Figure \ref{fig:comdiag} shows a generically non-commuting diagram relating all of these maps together.

The \emph{congruence subgroup of} $B_n$ \emph{mod $m$}, also called the \emph{level m congruence subgroup}, is denoted $B_n[m]$ and is the kernel of the integral Burau representation reduced mod $m$ (follow the top edge along the diagram in Figure \ref{fig:comdiag}),
\[B_n[m]:=ker(r_m\circ\rho_-).\]
 Congruence subgroups of $B_n$ are finite index subgroups that have been well studied by many authors, e.g. A'Campo \cite{ACampo}, Arnol'd \cite{A68}, Assion \cite{Assion},  Brendle \cite{brendle}, Brendle--Margalit \cite{brendleM18}, Bloomquist-Patzt-S. \cite{BPS}, Kordek--Margalit \cite{KM19}, McReynolds \cite{McReynolds}, Nakamura \cite{Nakamura}, and Stylianakis \cite{S18}. 
 In particular, Arnol'd proved that the congruence subgroup mod $2$ is the pure braid group as a subgroup of $B_n$ \cite{A68}.

\begin{figure}
    \centering
\begin{tikzcd}[sep=1.25cm]
&& GL_n(\mathbb{Z}) \ar[dr,"r_m"]
&
&[1.5em] \\
B_n  \ar[r, "\rho"] \ar[urr, "\rho_-", bend left=5, dashed] \ar[drr, "\rho_+", dashed, bend right=5] \ar[d, "\pi"] \ar[drr, "\circlearrowleft" near start, phantom, bend right=20] \ar[d, "\pi"]
& GL_n(\mathbb{Z}[t,t^{-1}])\ar[ur,swap,  "t=-1"]\ar[dr, "t=1"]
& \stackrel{\text{not generally}}{\text{\footnotesize{commutative}}}& GL_n(\mathbb{Z}_m)\\
S_n \ar[rr,swap,  "\phi:=\text{permutation rep.}"]&& GL_n(\mathbb{Z})\ar[ur, swap, "r_m"] 
&
& 
\end{tikzcd}
\caption{Diagram that is commutative exactly when $m=1,2$.}\label{fig:comdiag}
\end{figure}
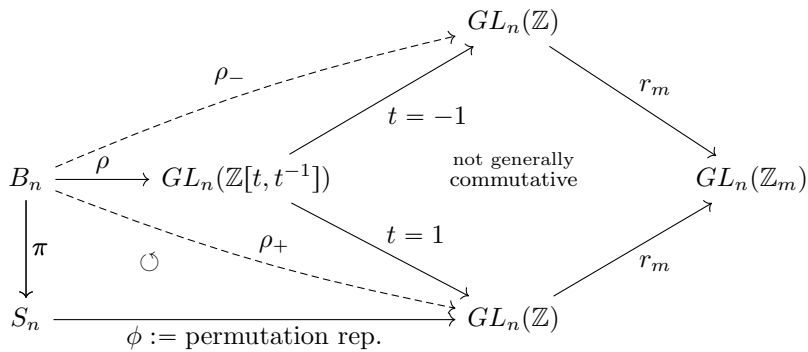

 \begin{thm}[Arnol'd, 1968]\label{thm:arnold} $B_n[2]=P_n$.\end{thm}

 The diagram in Figure \ref{fig:comdiag} recovers a simple algebraic proof of Arnol'd's result.
 This algebraic proof is likely known to experts, but, to the authors' knowledge, does not appear in the literature so we provide the proof here.
 Additionally, this algebraic proof generalizes to virtual congruence subgroups, in Section \ref{sec:virtual}, where Arnol'd's topological proof does not.

\begin{proof}[Proof of Theorem~\ref{thm:arnold}]
     After reducing mod 2, $-1=1$. So specializing $\rho$ at $t=1$ and then reducing mod 2 is the same as specializing $\rho$ at $t=-1$ and then reducing mod 2, $r_2\circ \rho_+=r_2\circ \rho_-$. Thus, the diagram in Figure \ref{fig:comdiag} is commutative when $m=2$, and $\ker(r_2\circ \rho_+)=\ker(r_2\circ \rho_-)=B_n[2]$.
     
     Since the bottom square of the diagram also commutes, we have that $\rho_+=\phi\circ \pi$, where $\phi$ is the permutation representation of $S_n$. 
     The map $\phi$ is injective with image the permutation matrices of $GL_n(\mathbb{Z})$, all of which have entries of 0 or 1. Reducing a permutation matrix mod 2 has no effect, so we get that $r_2\circ\phi$ is also injective.
     Finally, we see that $\ker(r_2\circ \phi \circ \pi)=\ker(\pi)=P_n$. Thus,
     \[B_n[2]=\ker(r_2\circ\rho_-)=\ker(r_2\circ \rho_+)=\ker(r_2\circ\phi\circ \pi)=\ker(\pi)=P_n.\]
 \end{proof}

Notice that for each generator $\sigma_i$, its square $\sigma_i^2$ is a pure braid and so $\sigma_i^2\in B_n[2]$. In fact, this pattern is true for all $m$, that is $\sigma_i^m\in B_n[m]$, \cite[R3: Powers of Dehn Twists]{S18}.
 Two other noteworthy facts about congruence subgroups are that $B_n[m]\cap B_n[\ell]=B_n[lcm(m,\ell)]$ and $B_n[m]\cdot B_n[\ell]=B_n[gcd(m,\ell)]$, which are proved in various forms in  \cite[Lemma 6.2]{S18} \cite[Proposition 3.1]{BPS}.

One last observation worth noting is that the kernel of Burau (unspecialized) is contained in the intersection of all the congruence subgroups of $B_n$. In particular, all braids in the kernel of Burau must be pure.

\section{Congruence Subgroups in the Virtual Setting}\label{sec:virtual}

There are many generalizations of the braid group from both an algebraic and topological point of view.
In this section we will look at the virtual braid group and extend the notion of congruence subgroups to this larger group.
We will refer to ``the braid group" as \emph{the classical braid group} to help clarify the difference between the classical and virtual settings.

The \emph{virtual braid group}, $vB_n$, is an extension of the classical braid group by adding virtual crossings. 
A \emph{virtual crossing} is a permutation of the strands --there is no sense of over or under-- and is depicted by a crossing with a circle, as seen in Figure \ref{fig:virtualbraids}.
Algebraically, these virtual crossings are transpositions in the symmetric group, so they will be denoted by $\tau_i$.
 A presentation for $vB_n$ is as follows. $vB_n$ is generated by $\sigma_1,\cdots, \sigma_{n-1}$ and $\tau_1,\cdots, \tau_{n-1}$, and subject to the following relations:

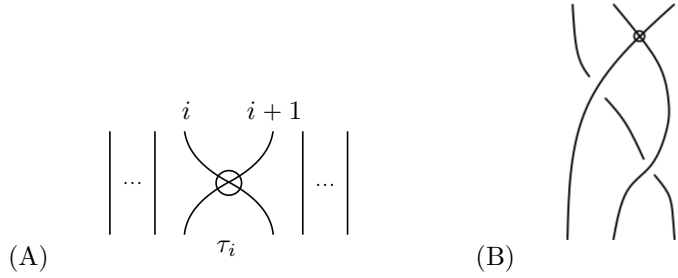
\begin{figure}
\centering
\begin{picture}(220,105)
\put(-20,15){\includegraphics[scale=.1]{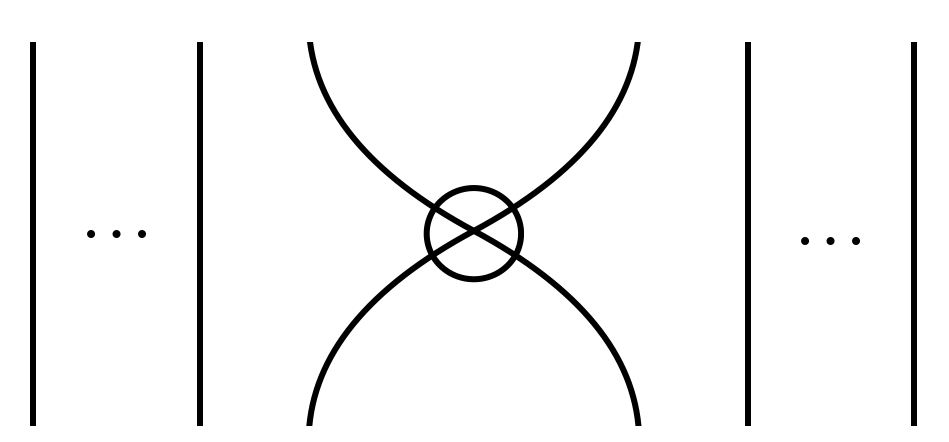}\hspace{.6cm}}
\put(150,10){\includegraphics[scale=.3]{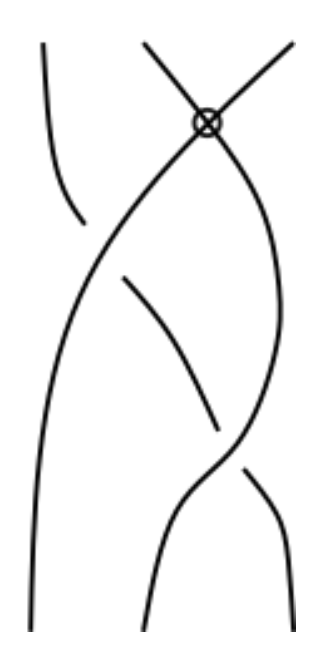}\hspace{.7cm}}
\put(-55,5){(A)}
\put(120,5){(B)}
\put(23,10){$\tau_i$}
\put(10,60){$i$ \hspace{.5cm} $i+1$}
\end{picture}
\caption{(A) The virtual crossing generator $\tau_i$. (B) An example virtual braid with two classical and one virtual crossings. }  \label{fig:virtualbraids}
\end{figure}

\begin{enumerate}
\item Classical relations: (the $\sigma_i$'s generate the classical braid group inside $vB_n$)
\begin{align*}
    \sigma_i\cdot \sigma_{i+1}\cdot \sigma_i &= \sigma_{i+1}\cdot \sigma_{i}\cdot \sigma_{i+1} &   \\
    \sigma_i\cdot \sigma_j &= \sigma_j\cdot \sigma_i & \text{for }| j-i|> 1
\end{align*}

\item Virtual relations: (the $\tau_i$'s generate the symmetric group inside $vB_n$)
\begin{align*}
    \tau_i^2&=id\\
    \tau_i\cdot \tau_{i+1}\cdot \tau_i &= \tau_{i+1}\cdot \tau_{i}\cdot \tau_{i+1}\\
    \tau_i\cdot \tau_j &= \tau_j\cdot \tau_i &\text{for }| j-i|> 1
\end{align*}

\item Mixed relations:
\begin{align*}
    \tau_i\cdot \sigma_{i+1}\cdot \tau_i &= \tau_{i+1}\cdot \sigma_{i}\cdot \tau_{i+1}\\
    \sigma_i\cdot \tau_j &= \tau_j\cdot \sigma_i &\text{for }| j-i|> 1
\end{align*}

\end{enumerate}
These relations come from the virtual Reidemeister moves for virtual knots \cite{Kauff1,Kauff2}, and there is a virtual analog of Alexander's theorem that states every virtual knot can be made as the closure of a virtual braid \cite{Kam}. Despite the mixed relations, $B_n$ embeds into $vB_n$, as was first proven by Kamada \cite{Kamada2004}; see also Bellingeri-Paris \cite{BP}.

There are two important maps from $vB_n$ to the symmetric group $S_n$. Firstly, in the classical setting,  the level 2 congruence subgroup of $B_n$ is the subgroup of pure braids, $B_n[2]=P_n$, where $P_n$ is the kernel of the map $\pi:B_n\rightarrow S_n$ given by $\pi(\sigma_i)=\tau_i$.
The map $\pi$ extends to the virtual braid group by acting as the identity on virtual crossings, that is $\pi_v:vB_n\rightarrow S_n$ is defined by $\pi_v(\sigma_i)=\tau_i$ and $\pi_v(\tau_i)=\tau_i$.
The kernel of $\pi_v$ is called the \emph{pure virtual braid group} and denoted here by $vP_n$,  \cite{B04}.
The second useful map is $\pi_K:vB_n\rightarrow S_n$ which maps $\pi_K(\sigma_i)=id$ and $\pi_K(\tau_i)=\tau_i$. The kernel of $\pi_K$ is denoted $kvB_n$ and is called the \emph{kure virtual braid group} and is the normal closure of $B_n$ inside of $vB_n$.
Unlike the classical setting, both maps $\pi_v$ and $\pi_K$ split $vB_n$ into two semi-direct products with $S_n$,  $vB_n=vP_n\rtimes S_n$ and $vB_n=kvB_n\rtimes S_n$, see \cite{B04,BP}.

Vershinin showed that the Burau representation extends to the virtual braid group in a natural way \cite{VER}, see also \cite{PS}, and in this paper we  work only with extending the integral Burau representation. 
As was discussed in Section \ref{sec:classical}, the Burau representation $\rho:B_n\rightarrow GL_n(\mathbb{Z}[t,t^{-1}])$ when specialized at $t=1$ yields the permutation representation of the symmetric group.
The virtual integral Burau representation on $vB_n$ will be denoted by $\rho_v$ where $\rho_v:vB_n\rightarrow GL_n(\mathbb{Z})$ is given by 
\begin{align*}
     \rho_v(\sigma_i)&=\rho|_{t=-1}(\sigma_i)=\rho_-(\sigma_i) &\text{  the usual integral Burau representation of } B_n\subset vB_n \\
     \rho_v(\tau_i):&=\rho|_{t=1}(\sigma_i)=\rho_+(\sigma_i) &\text{ the permutation representation of } S_n\subset vB_n. 
\end{align*}

Just as in the classical setting, we can define the \emph{congruence subgroup of} $vB_n$ mod $m$, or also known as the \emph{level $m$ congruence subgroup of} $vB_n$, as the kernel of the virtual integral Burau representation of $vB_n$ followed by $r_m$, reduction mod $m$.
This congruence subgroup will be denoted by $vB_n[m]$.
\[vB_n[m]:=\ker(r_m\circ\rho_v)\]

\begin{remk}\label{weld}
    Direct computation shows that $\rho_v$ factors through the quotient of $vB_n$ to the welded braid group, $wB_n$.  We have chosen to phrase all results in terms of the virtual braid group, but analogous results hold for the welded braid groups using identical proofs. Note that the normal closure of the additional relations  
    \begin{equation} \tau_i\sigma_{i+1}\sigma_i=\sigma_{i+1}\sigma_i\tau_{i+1}\end{equation}\label{eq:OCC} 
live entirely in the kernel of $\rho_v$.  An area of future work is to explore ways to bring the topological viewpoint of the welded braid group and it's isomorphism to the loop braid group into the conversation of congruence.
We appreciate the anonymous referee pointing out that, in contrast, $\rho_v$ does not factor through the unrestricted virtual braid group.
\end{remk}

By definition, $\rho_v$ restricted to $B_n$  as subset of $vB_n$ is the classical integral Burau representation $\rho_-$ on $B_n$. 
So we can see the containment $B_n[m]\subseteq vB_n[m]$ right away. 
The first question that arises is have we gained any new elements  by passing to the virtual congruence subgroup, or is $B_n[m] = vB_n[m]$? We quickly observe that $B_n[m]$ is a proper subgroup of $vB_n[m]$, as shown in the next Proposition.

\begin{prop}
$ B_n[m]\subsetneq vB_n[m].$
\end{prop}

\begin{proof}
    Note that $\sigma_i\tau_i\sigma_i\tau_i\in vB_n[m]$ for all $m$, as verified through direct matrix computation, and $\sigma_i\tau_i\sigma_i\tau_i$ is not a classical braid.
\end{proof}

We know $B_n[m]$ is a proper subgroup of $vB_n[m]$, but how much bigger is $vB_n[m]$? Is it possible that $vB_n[m]$ is the entire virtual braid group; does $vB_n[m]=vB_n$? This answer is again, no.

\begin{prop}
For $m>1$, a non-trivial braid in $vB_n[m]$ must contain a classical crossing, and thus $vB_n[m]\subsetneq vB_n$.
\end{prop}

\begin{proof}
    The virtual integral Burau representation restricted to $S_n$ as a subgroup of $vB_n$ is the permutation representation, which is injective on $S_n$. 
    Permutation matrices have entries that are 0's and 1's, which remain unchanged when reduced mod $m$, for any integer $m>1$.
    So, $r_m\circ \rho_v$ is injective on $S_n$, and $S_n=\langle \tau_1,\cdots,\tau_{n-1}\rangle$ are the only braids in $vB_n$ without a classical crossing. 
\end{proof}

So, we finally have that $B_n[m]\subsetneq vB_n[m]\subsetneq vB_n$
where $vB_n[m]$ is normal in $vB_n$, but $B_n[m]$ is not normal in either $vB_n[m]$ or $vB_n$. It is natural to question if  the normal closure of $B_n[m]$ inside $vB_n$ is $vB_n[m]$. 

\begin{prop}
    The normal closure of $B_n[m]$ inside $vB_n$ is not $vB_n[m]$. \footnote{We thank our anonymous referee for pointing out this argument to us.}
\end{prop}

\begin{proof}
Let $g=\tau_i\sigma_{i+1}\sigma_{i}\tau_{i+1}\sigma_i^{-1}\sigma_{i+1}^{-1}$ and recall the map $\pi_K:B_n\rightarrow S_n$ with kernel $kvB_n=\langle \langle B_n\rangle\rangle$. Since $\pi_K(g)=\tau_i\tau_{i+1}\neq id$, we have that $g\notin kvB_n=\langle \langle B_n\rangle\rangle$. Since $B_n[m]\subset B_n$ and $g\notin\langle \langle B_n\rangle\rangle$, then $g\notin\langle \langle B_n[m]\rangle\rangle$.
However, by Remark $\ref{eq:OCC}$, $g$ is in the kernel of $\rho_v$ and so $g\in vB_n[m]$ for every $m$.
\end{proof}

\begin{question}\label{ques:normalclosure}
    What is the normal closure of $B_n[m]$ inside $vB_n$?  More generally, what is a normal generating set for $vB_n[m]$? 
\end{question}

Turning our focus to the level 2 congruence subgroup, 
in the classical setting, we previously discussed the equality $B_n[2]=P_n$, where $P_n$ is the kernel of the map $\pi:B_n\rightarrow S_n$ given by $\pi(\sigma_i)=\tau_i$.
Recall that $\pi$ extends to the virtual braid group by $\pi_v:vB_n\rightarrow S_n$ is defined by $\pi_v(\sigma_i)=\tau_i$ and $\pi_v(\tau_i)=\tau_i$, with $\ker(\pi_v)=vP_n$  and  $vB_n=vP_n\rtimes S_n$,  see \cite{B04,BP}.
Using almost exactly the same argument as given in the proof of Theorem \ref{thm:arnold}, we show that the level 2 congruence subgroup of $vB_n$ is the pure virtual braid group $vP_n$.

\begin{thm}\label{thm:vpure}
    $vB_n[2]=vP_n$.
\end{thm}

\begin{proof}
    Let $\phi$ be the permutation representation of $S_n$.
Recall from the proof of Theorem \ref{thm:arnold} that $r_2\circ\rho_+(\sigma_i)=r_2\circ\phi\circ \pi(\sigma_i)$. Since $\pi=\pi_v|_{B_n}$ then $\phi\circ \pi(\sigma_i)=\phi\circ \pi_v(\sigma_i)$, which shows that  $r_2\circ\rho_+(\sigma_i)=r_2\circ \phi\circ \pi_v(\sigma_i)$.
\begin{center}

\adjustbox{center}{%
\begin{tikzcd}[sep=1.25cm]
vB_n \arrow[r, "\rho_v" ] \arrow[d, swap, "\pi_v"] \arrow[dr,phantom, "\circlearrowleft"]& GL_n(\mathbb{Z}) \arrow[d, "r_2" ] \\
S_n\arrow[r, swap, "r_2\circ \phi" ]
& GL_n(\mathbb{Z}/2\mathbb{Z})
\end{tikzcd}
}
\end{center}
The following identities show that the diagram above commutes.
\begin{align*}
    r_2\circ\rho_v(\sigma_i)&=r_2\circ \rho_-(\sigma_i)\stackrel{-1=1 (mod 2)}=r_2\circ\rho_+(\sigma_i)=r_2\circ \phi\circ \pi_v(\sigma_i)\\
 r_2\circ\rho_v(\tau_i)&=r_2\circ\rho_+(\sigma_i)=r_2\circ \phi\circ \pi_v(\sigma_i)\stackrel{\pi_v(\sigma_i)=\pi_v(\tau_i)}{=}r_2\circ \phi\circ \pi_v(\tau_i)
    \end{align*}

    Since $r_2\circ \phi$ is injective, then $\ker(\pi_v)=\ker(r_2\circ \phi\circ \pi_v)$. Thus,
    \[vP_n=\ker(\pi_v)=\ker(r_2\circ \phi\circ \pi_v)=\ker(r_2\circ\rho_v)=vB_n[2].\]

\end{proof}

\begin{cor}\label{cor:semi}
    $vB_n\cong vB_n[2]\rtimes (vB_n/vB_n[2])$.
\end{cor}

\begin{proof}
    $S_n\cong vB_n/vP_n$ and $vB_n=vP_n\rtimes S_n$.
\end{proof}

Corollary \ref{cor:semi} highlights a key difference between $B_n$ and $vB_n$. Since $B_n$ is torsion free, neither $S_n$ nor $B_n/B_n[m]$ are subgroup of $B_n$. Since $vB_n$ has torsion, there is potential for $vB_n/vB_n[m]$ to be a subgroup of $vB_n$ for $m$'s other than 1 and 2. This leads to an interesting question.

\begin{question}\label{ques:semidirectAsubgroup}
For what values of $m$ is $vB_n/vB_n[m]$ a subgroup of $vB_n$? Said another way, for what values of $m$ is  $vB_n\cong vB_n[m]\rtimes (vB_n/vB_n[m])$?
\end{question}

\begin{prop}(An answer to Question \ref{ques:semidirectAsubgroup} for $n\geq 5$)\footnote{We thank our anonymous referee for pointing out this argument to us.}
    For $n\geq 5$ and $m\neq 1$, then $vB_n/vB_n[m]$ embeds into $vB_n$ if and only if $m=2$.
\end{prop}

\begin{proof}
Let $G=vB_n/vB_n[m]$ and suppose that $G$ embeds as a subgroup of $vB_n$. Notice $G$ is a finite group in $vB_n$. 
The restriction of $r_m\circ \rho_v$ to $S_n$ is injective, hence the index of $vB_n[m]=\ker(r_m\circ \rho_v)$ in $vB_n$ is at least $n!$, and so $|G|\leq n!$.
Furthermore, since $kvB_n$ is torsion free \cite{GP12}, $\pi_K$ is injective on $G$ which forces $|G|\leq n!$. Thus $|G|=n!$ and $G$ is isomorphic to $S_n$.

Bellengeri-Paris proved that when $n\geq 5$, the only homomorphisms from $vB_n$ to $S_n$ are $\pi_v$ or $\pi_K$ \cite{BP}[see Theorem 2.1]. So, considering the quotient map $vB_n\rightarrow (vB_n/vB_n[m]=G\cong S_n)$ with kernel $vB_n$, it must be the case that $vB_n[m]=\ker(\pi_v)=vP_n$ or $vB_n[m]=\ker(\pi_K)=kvB_n$.
However, $\sigma_i\in kvB_n$ by definition of $\pi_K$, but $\sigma_i\notin vB_n[m]$ for any $m$ (this can be verified by direct computation or the fact that $\sigma_i\notin B_n[m]$). So $vB_n[m]\neq kvB_n$ and it must be the case that $vB_n[m]=vP_n$. Now $\sigma_i\tau_1\in vP_n$, but by direct computation one can show that $\sigma_1\tau_1\in vB_n[m]$ if and only if $m=2$.
\end{proof}

As a start to answer this question for all $n$, it would be helpful to identify the image of virtual integral Burau, and then perhaps classify $vB_n/vB_n[m]$ as was done in for the classical case in \cite{ACampo, BPS} and others. Additionally, it would be helpful to have a presentation for the subgroup $vB_n[m]$.

\begin{question}\label{ques:whatstheimage}
    What is the image of the virtual integral Burau representation?
\end{question}

Next, we look to relationships between congruence subgroups.
In the classical setting, $B_n[m]\cap B_n[\ell]=B_n[lcm(m,\ell)]$. We prove the analogous is true in the virtual setting by following the classical proof in \cite{BPS}.

\begin{prop}
    $vB_n[m]\cap vB_n[\ell]=vB_n[lcm(m,\ell)]$.
\end{prop}

\begin{proof}
A braid $\beta$ is in  $vB_n[lcm(\ell,m)]$ if and only if $\rho_v(\beta)=Id_n+X$ where $X$ is an $n\times n$ matrix such that all entries of $X$ are divisible by $lcm(\ell,m)$. This is if and only if $\rho_v(\beta)=Id_n+X$ such that all entries of $X$ are divisible by both $m$ and $\ell$, which is if and only if $\beta\in vB_n[\ell]\cap vB_n[m]$.
\end{proof}

For one last comparison between classical and virtual congruence subgroups, in the classical setting, $B_n=B_n[m]\cdot B_n[\ell]$ when  $m$ and $\ell$ are relatively prime \cite{S18,BPS}.
The standard proof of this identity is to show that each generator of $B_n$ is contained in the product $B_n[m]\cdot B_n[\ell]$ by relying on the fact that $\sigma_i^m\in B_n[m]$. 
If we tried to repeat this argument for $vB_n$, we would also need to show that $\tau_i\in vB_n[m]\cdot vB_n[\ell]$.
When $m=2$, we can show $\tau_i\in vB_n[2]\cdot vB_n[\ell]$.
However, if neither $m$ nor $\ell$ is 2, then $\tau_i\notin vB_n[m]\cdot vB_n[\ell]$, regardless of whether $m$ and $\ell$ are relatively prime.

\begin{prop}\label{prop:relprimenot2}
     $vB_n= vB_n[m]\cdot vB_n[\ell]$ if and only if $m=2$ and $\ell$ is odd.
\end{prop}

\begin{proof}

 First suppose that neither $m$ nor $\ell$ is $2$.  We proceed by using the determinant map together with $\rho_v$.  First we note that $det(\rho_v(\sigma_i))=1$ and $det(\rho_v(\tau_i))=-1$ for all $i$.  This implies the following diagram commutes
 
\[\begin{tikzcd}
     vB_n\arrow[d,"\pi_K"] \arrow[r,"\rho_v"] & GL_n(\mathbb{Z})\arrow[d,"det"] \\
     S_n\arrow[r,"sgn"] & \mathbb{Z}/2\mathbb{Z}.
 \end{tikzcd}\]

 We note that $\pi_K$ is not the map to the symmetric group used to define $vP_n$, but rather the map defined on generators by $\pi_K(\sigma_i)=id$ and $\pi_K(\tau_i)=\tau_i$ with kernel $kvB_n$. 
 We have that $vB_n=kvB_n\rtimes S_n$, and $ker(sgn\circ \pi_K)=kvB_n\rtimes A_n\subsetneq vB_n$.  As such $ker(det\circ \rho_v)\subsetneq vB_n$.  Then for any $b_m\in vB_n[m]$, as $r_m\circ \rho_v(b_m)=Id$, we have $det\circ r_m\circ \rho_v(b_m)=1$.  
 Using that $m\neq 2$ and thus $-1\neq 1$, this implies $det\circ \rho_v(b_m)=1$, and $vB_n[m]\subset ker(det\circ \rho_v)$.  
 Similarly $vB_n[\ell]\subset ker(det\circ \rho_v)$, which implies $vB_n[m]\cdot vB_n[\ell]\subseteq ker(det\circ \rho_v)\subsetneq vB_n$. (In particular, $\tau_i$ will never be in the product $vB_n[m]\cdot vB_n[\ell]$.)

 Now assume that $m=2$ and $\ell$ is odd.  
 Then by Bezout's Lemma there exist integers $a$ and $b$ such that $2a+\ell b=1$.  Then noting that $\sigma_i^2\in vB_n[2]$, $\sigma_i\tau_i\in vB_n[2],$ and $\sigma_i^\ell\in vB_n[\ell]$ we have
 \[\sigma_i=\sigma_i^{2a+\ell b}=(\sigma_i^2)^a(\sigma_i^\ell)^b\in vB_n[2]\cdot vB_n[\ell]\]
 \[\tau_i=\tau_i\sigma_i((\sigma_i^2)^a(\sigma_i^\ell)^b)^{-1}=(\tau_i\sigma_i(\sigma_i^{2})^{-a})(\sigma_i^{\ell})^{-b} \in vB_n[2] \cdot vB_n[\ell].\]
\end{proof}

In the classical setting, $B_n[gcd(m,\ell)]=B_n[m]\cdot B_n[\ell]$ for any choice of $m$ and $\ell$. Proposition \ref{prop:relprimenot2} shows that the virtual analog of this identity is not true in general; when $gcd(m,\ell)=1$, $vB_n[1]=vB_n[gcd(m,\ell)]=vB_n[m]\cdot vB_n[\ell]$ only when $m=2$ and $\ell$ is odd.
It is always true that $vB_n[m]\cdot vB_n[\ell]$ is contained in $vB_n[gcd(m,\ell)]$ as for any $\alpha\in vB_n[m]$ and $\beta\in vB_n[\ell]$,
\begin{align*}
\rho_v(\alpha)&=Id+mX\\
\rho_v(\alpha)&=Id+\ell Y\\
\rho_v(\alpha\beta)&=Id+mX+\ell Y+m\ell XY=Id+gcd(m,\ell)Z
\end{align*}
which shows $\alpha\beta\in vB_n[gcd(m,\ell)]$. As a product of two normal subgroups, $vB_n[m]\cdot vB_n[\ell]$ is therefore always a subgroup of $vB_n[gcd(m,\ell)]$. The question is if the subgroup is proper, or not.

\begin{question}\label{ques:indexofprod}
      What is the index of $ vB_n[m] \cdot vB_n[\ell] $ in $vB_n[gcd(m,\ell)]$? If $gcd[m,\ell]\neq 1$, is $ vB_n[m] \cdot vB_n[\ell] =vB_n[gcd(m,\ell)]$?
\end{question}

To give more perspective on why these questions are interesting beyond their classical counterparts, in the final section of this paper we answer many of these questions for the 2 strand virtual braid group, $vB_2$.
Similar to the classical braid group, the 2 strand virtual braid group is too small to see generalizable behavior to higher values of $n$.
However, even with such a small value of $n$, the reader can see the increased complexity that comes from working in the virtual setting instead of the classical one.

 \subsection{Answering Questions \ref{ques:normalclosure}-\ref{ques:indexofprod} for $vB_2$.}

 When $n=2$, the presentation of the virtual braid group becomes considerably simpler as the classical relations, mixed relations, and virtual relations (other than $\tau_1^2=id$) all become vacuous.  This gives the following presentation,
 \[vB_2\cong \langle \sigma,\tau|\tau^2=id\rangle\cong \mathbb{Z}\Asterisk \mathbb{Z}/2\mathbb{Z}.\]

From this presentation, we can see the infinite Dihedral group, with presentation $D_\infty\cong \langle s,t|t^2,stst\rangle$, is a quotient of $vB_2$. Also, the finite Dihedral group of  order $2m$, presented by $D_{2m}\cong \langle s,t|t^2,stst,s^m\rangle$, is a further quotient of $vB_2$.
In following theorem, we show that the image of the virtual integral Burau representation composed with reduction $\mod m$ of $vB_2$ is isomorphic to $D_{2m}$ as a subgroup of $GL(2, \Z/m\Z)$ using the commutative diagram in Figure \ref{fig:Dyhedralcomdiag}.
For notational clarity, we let  \[A:=\rho_{v}(\sigma)=\begin{pmatrix}
    2 & -1 \\
    1 & 0
\end{pmatrix}, \qquad B:= \rho_{v}(\tau)=\begin{pmatrix}
    0 & 1 \\
    1 & 0
\end{pmatrix}.\]

\begin{thm}\label{thm:Dyhedralcomdiag}
The diagram in Figure \ref{fig:Dyhedralcomdiag} commutes, and the map
    $f:D_\infty \hookrightarrow \mathrm{GL}(2,\mathbb{Z})$
    defined by $s\mapsto A$ and $t\mapsto B$
    is an injective group homomorphism. Moreover the map 
    $f_m:D_{2m} \hookrightarrow \mathrm{GL}(2,\mathbb{Z}/m\mathbb{Z})$ defined by $s\mapsto [A]_m$ and $t\mapsto [B]_m$
    is also an injective group homomorphism as long as $m\neq 2$.
\end{thm}

\begin{figure}
\[\begin{tikzcd}
 & \mathrm{GL}(2,\mathbb{Z})\arrow[r,two heads,"r_m"] & \mathrm{GL}(2,\mathbb{Z}/m\mathbb{Z}) \\
  & \langle A,B\rangle \arrow[r,two heads]\arrow[u,hook] & \langle [A]_m,[B]_m\rangle\arrow[u,hook]\\
vB_2\arrow[ur,"\rho_{v}"]\arrow[dr] &  &  \\
 & D_\infty\arrow[uu,hook,"f"]\arrow[r,two heads] & D_{2m}\arrow[uu,hook,"f_m"] \\
\end{tikzcd}\]
\caption{ Commutative diagram for Theorems \ref{thm:Dyhedralcomdiag} and \ref{thm:normalclosure}.}
\label{fig:Dyhedralcomdiag}
\end{figure}

\begin{proof}
Direct computation shows that $ABAB=Id$, which shows $f$ and $f_m$ are group homomorphisms.   
 The diagram in Figure \ref{fig:Dyhedralcomdiag} commutes on generators through the identification of $\sigma$ with $s$ and $\tau$ with $t$.  

Notice in $D_\infty$, the fact that $stst=e$ (which implies $ts=s^{-1}t$) combined with $t^2=e$ gives that every element of $D_\infty$ admits a unique expression of the form $s^at^b$ where $a\in \mathbb{Z}$ and $b\in \mathbb{Z}/2\mathbb{Z}$.

Then direct computation gives $f(s^a)=\begin{pmatrix}
    1+a & -a \\
    a & 1-a\\
\end{pmatrix}$ and $f(s^at)=\begin{pmatrix}
    -a & 1+a \\
    1-a & a\\
\end{pmatrix}$ which are distinct for all values of $a\in \Z$, and hence shows that $f$ is injective. 

Similarly we see the result for $f_m$ if $m\neq 2$, where the values of $a$ are contained to $1,...,m-1$.
    
\end{proof}

\begin{cor}\label{cor:imagen=2}(An answer to Question \ref{ques:whatstheimage} for n=2.)
  The image of the virtual integral Burau representation of $vB_2$ is the infinite Dihedral group, $\rho_v(vB_2)\cong D_\infty$.
  \end{cor}
\begin{proof}
    Since $f$ is injective, $D_\infty\cong im(f)=\langle A,B\rangle=im(\rho_v)$.
\end{proof}

Following the commutative diagram in Figure \ref{fig:Dyhedralcomdiag}, we know that $r_m\circ \rho_v$ is equal to the surjection $f_m\circ q_m$, where 
    \[q_m:vB_2\rightarrow D_{2m}\]
    is defined on generators as $q_m(\sigma)=s$ and $q_m(\tau)=t$.  The injectivity of $f_m$ implies that $vB_2[m]$ is exactly the kernel of $q_m$ and we get the following short exact sequence.

    \begin{equation*}\tag{$\dagger$}1\rightarrow vB_2[m]\rightarrow vB_2\xrightarrow[]{q_m} D_{2m}\rightarrow 1 \end{equation*}

Working from the corresponding group presentations in the commutative diagram from Figure \ref{fig:Dyhedralcomdiag} we can find normal generators for $vB_2[m]$.
\begin{thm}\label{thm:normalclosure}
    (An answer to Question \ref{ques:normalclosure}(B.) for n=2.) For any $m\neq 2$, we have 
    \[vB_2[m]\cong \langle\langle \sigma\tau\sigma\tau, \sigma^m\rangle\rangle.\]
\end{thm}

\begin{proof}
     
    Following the known group presentations, we have 
    \[q_m:\langle \sigma,\tau|\tau^2\rangle\rightarrow \langle s,t|t^2, stst, s^m\rangle,\]
    and as such, we know $vB_2[m]=ker(q_m)=\langle\langle \sigma\tau\sigma\tau,\sigma^m\rangle\rangle$.
\end{proof}

\begin{thm}\label{thm:notnormalclosure} (An answer to \ref{ques:normalclosure}(A.) for n=2.) For all $m\neq 2$ we have $vB_2[m]\neq \langle \langle B_2[m]\rangle\rangle$.
\end{thm}
\begin{proof}
    First note that $B_2=\langle \sigma\rangle$ and $B_2[m]=\langle \sigma^m\rangle$.  
    From the short exact sequence ($\dagger$), we see that $vB_2/vB_2[m]\cong D_{2m}$.  Meaning in particular that $vB_2[m]$ is finite index in $vB_2$.

    Now we look to analyze $vB_2/\langle\langle B_2[m]\rangle\rangle=vB_2/\langle \langle \sigma^m\rangle \rangle\cong \langle \sigma,\tau|\tau^2=\sigma^m=id\rangle$.  Thus we observe that $\langle\langle B_2[m]\rangle\rangle$ is infinite index in $vB_2$.

    Thus $vB_2[m]\neq \langle \langle B_2[m]\rangle\rangle$.
    
\end{proof}

\begin{thm}\label{thm:n=2notsemidirect}
(An answer to Question \ref{ques:semidirectAsubgroup} when $n=2$.)    When $m\neq 2$, we have 
    \[vB_2\not\cong vB_2[m]\rtimes (vB_2/vB_2[m]).\]
\end{thm}
\begin{proof}

    From the short exact sequence ($\dagger$), we see that $vB_2/vB_2[m]\cong D_{2m}$.  
    Then using that $vB_2\cong \mathbb{Z}\Asterisk\mathbb{Z}/2\mathbb{Z}$ together with the Kurosh subgroup theorem we conclude that all finite order elements of $vB_2$ must be order $2$.  Thus $D_{2m}$ is not a subgroup of $vB_2$ and the result follows.
 \end{proof}

\bibliographystyle{plain}
\bibliography{main.bib}{}

\end{document}